\documentclass[12pt]{amsart}
\usepackage{amsfonts,amsmath,amssymb}

\newtheorem{theorem}{Theorem}[section]
\newtheorem{corollary}[theorem]{Corollary}
\newtheorem{lemma}[theorem]{Lemma}
\newtheorem{proposition}[theorem]{Proposition}
\newtheorem{conjecture}[theorem]{Conjecture}
\newtheorem{conventions}[theorem]{Conventions}
\theoremstyle{definition}

\theoremstyle{remark}
\newtheorem{remark}[theorem]{Remark}
\newtheorem{diagram}{Diagram}
\numberwithin{equation}{section}
\usepackage{url}
\usepackage{hyperref,amsthm,amssymb}
\hypersetup{colorlinks=true,linkcolor=blue,urlcolor=cyan} 
\usepackage[hyperpageref]{backref}
\usepackage{frcursive}
\usepackage{calrsfs}
\usepackage{mathrsfs}
\newcommand{\Q}{\mathbb{Q}}
\newcommand{\Z}{\mathbb{Z}}
\newcommand{\F}{\mathbb{F}}
\newcommand{\ds}{\displaystyle}
\newcommand{\ov}{\overline}
\newcommand{\wt}{\widetilde}
\newcommand{\wh}{\widehat} 
\newcommand{\sst}\scriptstyle
\newcommand{\ft}{\footnotesize}
\newcommand{\ns}\normalsize
\newcommand{\Tk}{{\bf \Sigma}_k}
\newcommand{\TK}{{\bf \Sigma}_K}
\newcommand{\tk}{{\bf T}_k}
\newcommand{\cl}{c\hskip-1pt{\ell}}
\newcommand{\order}{\raise0.8pt \hbox{${\sst \#}$}}
\newcommand{\lien}{\mathrel{\mkern-4mu}}
\newcommand{\too}{\relbar\lien\rightarrow}
\newcommand{\tooo}{\relbar\lien\relbar\lien\too}

\newcommand{\plus}{\ds\mathop{\raise 2.0pt \hbox{$\bigoplus$}}\limits}
\newcommand{\prd}{\ds\mathop{\raise 2.0pt \hbox{$\prod$}}\limits}
\newcommand{\sm}{\ds\mathop{\raise 2.0pt \hbox{$\sum$}}\limits}
\newcommand{\ffrac}[2]{\hbox{\ft $\displaystyle\frac{#1}{#2}$}}

\newcommand{\Gal}{{\rm Gal}}
\newcommand{\No}{{\rm N}}
\newcommand{\pr}{{\rm pr}}
\newcommand{\nr}{{\rm nr}}
\newcommand{\ram}{{\rm ram}}
\newcommand{\gen}{{\rm gen}}

\newcommand{\ab}{{\rm ab}}
\newcommand{\bp}{{\rm bp}}
\newcommand{\ta}{{\rm ta}}
\newcommand{\ppl}{{\rm ppl}}

\newcommand{\tor}{{\rm tor}}
\newcommand{\Hom}{{\rm H}}
\newcommand{\Ker}{{\rm Ker}}
\newcommand{\Cha}{{\rm III}}

\begin{document}

\title[Algorithmic complexity of Greenberg's conjecture]
{Algorithmic complexity \\ of Greenberg's conjecture}

\author{Georges Gras}

\address{Chemin de Ch\^ateau Gagni\`ere, Villa la Gardette,
38520 Le Bourg d'Oisans}
\email{g.mn.gras@wanadoo.fr}



\begin{abstract}
Let $k$ be a totally real number field and $p$ a prime. We show that the 
``complexity'' of Greenberg's conjecture ($\lambda = \mu = 0$) is of 
$p$-adic nature governed (under Leopoldt's conjecture) by the finite torsion 
group ${\mathcal T}_k$ of the Galois group of the maximal abelian 
$p$-ramified pro-$p$-extension of $k$, by means of images in ${\mathcal T}_k$
of ideal norms from the layers $k_n$ of the cyclotomic tower (Theorem \ref{t}). 
These images are obtained via the formal algorithm computing, by ``unscrewing'',
the $p$-class group of~$k_n$. Conjecture \ref{conj} of equidistribution 
of these images would show that the number of steps $b_n$ of the algorithms
is bounded as $n \to \infty$, so that Greenberg's conjecture, hopeless within 
the sole framework of Iwasawa's theory, would hold true ``with probability $1$''. 
No assumption is made on $[k : \Q]$, nor on the decomposition of $p$ in $k/\Q$.
\end{abstract}

\maketitle

\tableofcontents

\section{Introduction} 
Let $k$ be a totally real number field, $p \geq 2$ a prime number and
$S$ the set of $p$-places ${\mathfrak p} \mid p$ of $k$. Let $k_\infty$ be 
the cyclotomic $\Z_p$-extension of $k$ and $k_n$ the degree $p^n$ 
extension of $k$ in $k_\infty$.
Let ${\mathcal C}_k$ and ${\mathcal C}_{k_n}$ be the $p$-class groups 
of $k$ and $k_n$, respectively. 
We denote by ${\mathcal T}_k$ the torsion group of ${\mathcal A}_k := \Gal(H_k^\pr/k)$, 
where $H_k^\pr$ is the maximal abelian $S$-ramified  pro-$p$-extension of $k$
(i.e., unramified outside $S$), assuming the Leopoldt conjecture for $p$ in $k_\infty$.
The group ${\mathcal T}_k$ is closely related to the deep Tate--Chafarevich 
group (same $p$-rank):
$$\Cha_k^2 := {\rm Ker} \big [{\Hom}^2 ({\mathcal G}_{k,S},\F_p) \rightarrow
\oplus_{{\mathfrak p} \mid p} \, {\Hom}^2 ({\mathcal G}_{k_{\mathfrak p}},\F_p) \big], $$
where ${\mathcal G}_{k,S}$ is the Galois group of the maximal $S$-ramified 
pro-$p$-extension of~$k$ (hence ${\mathcal A}_k = {\mathcal G}_{k,S}^\ab$)
and ${\mathcal G}_{k_{\mathfrak p}}$ the local analogue over $k_{\mathfrak p}$;
but ${\mathcal T}_k$ is very easily computable and relates the $p$-class group 
and the $p$-adic regulator.

\smallskip
We call {\it Greenberg's conjecture for $k$ and $p$}, the nullity of the Iwasawa 
invariants $\lambda$, $\mu$ (see the origin of the conjecture 
in \cite[Theorems 1 and 2]{Gre}).
The main effective test for this conjecture is 
the criterion of Jaulent \cite[Th\'eor\`emes A, B]{J2} proving that the
conjecture is equivalent to the capitulation in $k_\infty$ of the logarithmic 
class group $\wt {\mathcal C}_k$ of $k$ (defined in \cite{J1} with PARI/GP
pakage in \cite{BJ}), an invariant
also related to $S$-ramification theory.\,\footnote{For more information on the
main pioneering works about {\it the practice} of this theory, see ``history of abelian 
$p$-ramification'' in \cite[Appendix]{Gra7} (e.g., Gras: ``Crelle's Journal'' 
(1982/83), Jaulent: ``Ann. Inst. Fourier'' (1984), Nguyen Quang Do: 
``Ann. Inst. Fourier'' (1986), Movahhedi ``Th\`ese'' (1988) and others). 
For convenience, we mostly refer to our book (2003/2005), which contains all 
the needed results in the most general statements. For more broad context 
about the base field and the set $S$, see \cite{Mai} and its bibliography.}
For specific cases of decomposition of $p$, as in \cite{Gre}, see \cite{Ng2}.

In our opinion, many aesthetic statements, equivalent to
Greenberg's conjecture, are translations of standard formalism of class field 
and Iwasawa's theories. In other words, {\it some ``non-algebraic'' $p$-adic 
aspects of the ``diophantine construction'' of the class groups at each layer $k_n$}, 
are not taken into account. We show how this construction works and study its 
arithmetic complexity by means of the number $b_n$ of steps of the algorithms
which become oversized in the tower as soon as $\lambda$ or $\mu$ are non-zero, 
suggesting the triviality of the algorithms for $n \gg 0$ (i.e., $b_n \leq 1$).

Our purpose has nothing to do with computational or theoretical approaches in the area
of the ``main theorem'' on abelian fields (analytic formulas, cyclotomic units, 
$L_p$-functions, etc.) as, for instance, the very many contributions 
(cited in our papers \cite{Gra3,Gra4}),
also giving computations and suggesting that equidistribution results 
may have striking consequences for the conjecture; our
viewpoint is essentially logical and based on the governing group ${\mathcal T}_k$,
because we have conjectured that ${\mathcal T}_k = 1$
for all $p \gg 0$, due to properties of $p$-adic regulators \cite{Gra6}
($p$-rationality of $k$, as defined in \cite{Mov} for such fields), 
which relativizes Greenberg's conjecture, obvious in that case.

In many papers, as in \cite{Gre}, the
decomposition of $p$ in $k/\Q$ plays a specific role, which is not necessary for us. 
We shall not put any assumption on the degree of $k$ nor on the decomposition 
of $p$ in $k/\Q$.

\begin{conventions} \label{conventions}
Subject to replace $k$ by a layer $K=k_{n_0}$ of $k_\infty = K_\infty$, 
one may assume, without any loss of generality, that $p$ is totally ramified in 
$K_\infty/K$ and is such that Iwasawa's formula for $\order {\mathcal C}_{k_n}$ 
holds true for all layers above $K$;
indeed, we have $\lambda(K) = \lambda(k)$, $\mu(K) = [K : k]\, \mu(k)$ and
$\nu(K)=\nu(k) + \lambda(k)\,n_0$. 
\end{conventions}

\section{Main results}
The results of the paper may be described as follows in two parts:

\smallskip
{\bf (A)} From results of \cite{Gra2,Gra3,Gra4}.
The formal algorithm, determining $\order {\mathcal C}_{k_n}$ (whence
giving the Iwasawa invariants), computes inductively 
the classical filtration $({\mathcal C}_{k_n}^i)_{i \geq 0}$, where 
${\mathcal C}_{k_n}^{i+1} / {\mathcal C}_{k_n}^i := 
({\mathcal C}_{k_n} / {\mathcal C}_{k_n}^i)^{G_n}$, for 
all $i \geq 0$ (${\mathcal C}_{k_n}^0=1$), where $G_n = \Gal(k_n/k)$. 
We have the decreasing $i$-sequence:
\begin{equation}\label{algo} 
\order  \big( {\mathcal C}_{k_n}^{i+1} / {\mathcal C}_{k_n}^i \big) = 
\frac{\order  {\mathcal C}_k}{\order \No_{k_n/k}( {\mathcal C}_{k_n}^i)} \cdot  
\frac{p^{n \cdot (\order S -1)}}
{(\Lambda_n^i : \Lambda_n^i \cap \No_{k_n/k}(k_n^\times))}, 
\end{equation}
\noindent
with the increasing $i$-sequence of groups $\Lambda_n^i$, from $\Lambda_n^0=E_k$:
\begin{equation}\label{represent}
\Lambda_n^i := \{x \in k^\times,\ \, (x) = \No_{k_n/k}({\mathfrak A}), 
\ \, \cl_{k_n}({\mathfrak A}) \in  {\mathcal C}_{k_n}^i \} .
\end{equation} 

Then ${\mathcal C}_{k_n}^{i+1} / {\mathcal C}_{k_n}^i$ in \eqref{algo} becomes trivial 
for some minimal $i=:b_n \geq 0$ (giving ${\mathcal C}_{k_n}^{b_n} = 
{\mathcal C}_{k_n}$) as soon as the two factors vanish. 
Thus the length $b_n$ of the algorithm depends on the decreasing evolution 
of the ``class factor'' $\ffrac{\order {\mathcal C}_k}{\order \No_{k_n/k}( {\mathcal C}_k^i)}$ 
dividing $\order {\mathcal C}_k$ and that of the ``norm factor'' $\ffrac{p^{n \cdot (\order S -1)}}
{(\Lambda_n^i : \Lambda_n^i \cap \No_{k_n/k}(k_n^\times))}$
dividing the order of a suitable quotient ${\mathcal R}_k^\nr$ of the 
normalized $p$-adic regulator ${\mathcal R}_k$ (defined in \cite[\S\,5]{Gra5}), 
related to the ramification of $p$ in $H_k^\pr/k_\infty$ 
(Theorem \ref{genus2}, Corollary \ref{fact}).
We prove in Theorem \ref{O}, under Conventions \ref{conventions},
the following inequalities (where $v_p$ is the $p$-adic valuation):
$$b_n \leq \lambda \cdot n + \mu \cdot p^n + \nu \leq 
v_p(\order {\mathcal C}_k \cdot \order {\mathcal R}_k^\nr) \cdot b_n, $$
giving ${\mathcal C}_k = {\mathcal R}_k^\nr =1
\Longleftrightarrow \lambda=\mu=\nu= 0 \Longleftrightarrow b_n=0$ for all $n$. 

\smallskip
Taking $k$ hight enough in the tower, Greenberg's conjecture is equivalent 
to $b_n \leq 1$ for all $n$ (Corollary \ref{bn=1}), which constitutes a spectacular 
algorithmic discontinuity compared to $b_n \to \infty$ if $\lambda$ or $\mu$ 
are non-zero. In an heuristic point of view, it is ``necessary'' that the algorithms
become limited, because of the unpredictable behavior of the class and norm factors. 

\medskip
{\bf (B)} One may replace, in \eqref{represent}, the ideal norms 
${\mathfrak a} = \No_{k_n/k}({\mathfrak A})$ by representatives
${\mathfrak t} \in I_{k_n} \otimes \Z_p$ ($I_{k_n}$ is the group 
of prime-to-$p$ ideals of $k_n$) whose Artin symbols are in 
${\mathcal T}_k$, {\it hence finite in number\,} (main Theorem \ref{t}); 
so, each step of the algorithm (i.e., the evolution of the class and norm factors) 
only depends on at most $\order {\mathcal T}_k$ possibilities, taking the class of the 
random ideal ${\mathfrak t}$, then computing Hasse's symbols
on $S$ of numbers $\tau \in k_n^\times \otimes \Z_p$ when 
${\mathfrak t}=(\tau)$ is principal, in other words, 
for this last case a classical situation 
involving random $\Z/p^n \Z$-matrices of symbols for which some 
equidistribution results are proven \cite[Section 6]{Sm}. 

\smallskip
Then, under the natural Conjecture \ref{conj} of independence and 
randomness of the data obtained, inductively, at each step of the algorithm, 
one would obtain that Greenberg's conjecture holds true with ``probability $1$'', 
suggesting possible analytic proof of this fact, using the powerful 
techniques used in \cite{KP,Sm} for degree $p$ cyclic extensions of $\Q$, but 
unfortunately, probably not a complete proof of Greenberg's conjecture.

\section{Abelian $p$-ramification and genus theories}\label{sub1}

\subsection{Abelian $p$-ramification -- The torsion group ${\mathcal T}_k$}
\label{notations}
Recall the data needed for the study of the Galois group 
${\mathcal A}_k$ of the maximal abelian
$p$-ramified pro-$p$-extension $H_k^\pr$ of $k$
and its torsion group ${\mathcal T}_k$ (under Leopoldt's conjecture).
Let $k'^\times$ be the subgroup of $k^\times$ of prime-to-$p$ elements:

\smallskip
\noindent
\quad $\bf (a)$ Let $E_k$ be the group of $p$-principal units 
$\varepsilon \equiv 1 \!\!\pmod {\prod_{{\mathfrak p} \in S} {\mathfrak p}}$ 
of~$k$. Let $U_k:=\hbox{$\bigoplus_{{\mathfrak p} \in S}$} U_{\mathfrak p}$ 
be the $\Z_p$-module of $p$-principal local units, where $U_{\mathfrak p}$ 
is the group of ${\mathfrak p}$-principal
units of the ${\mathfrak p}$-completion $k_{\mathfrak p}$ of $k$. 
Let $\mu_k$ (resp. $\mu_{\mathfrak p}$) be the 
group of $p$th roots of unity of $k$ (resp. $k_{\mathfrak p}$). Put
$W_k := \bigoplus_{{\mathfrak p} \in S} 
\mu_{\mathfrak p}$ and ${\mathcal W}_k := W_k/\mu_k$; thus,
${\mathcal W}_k = W_k$ for $p \ne 2$ and ${\mathcal W}_k = 
W_k / \langle \, \pm 1\, \rangle$ for $p=2$.

\smallskip
\noindent
\quad $\bf (b)$
Let $\iota :  k'^\times \otimes \Z_p \to U_k$ be the canonical 
surjective diagonal map. 
Let $\ov  E_k$ be the closure of 
$\iota E_k$ in $U_k$ and let $H_k^\nr$ be the $p$-Hilbert class field of~$k$.
By class field theory, $\Gal(H_k^\pr/k_\infty H_k^\nr) \simeq 
 \tor_{\Z_p}(U_k/\ov  E_k) = U_k^*/\ov  E_k$, where 
$U_k^* := \{u \in U_k, \,\ \No_{k/\Q}(u) \in \langle \,\pm 1\, \rangle \}$.

\smallskip
\noindent
\quad $\bf (c)$ Let ${\mathcal C}_k$ be the $p$-class group of $k$ and let:
\begin{equation}\label{u*}
{\mathcal R}_k := \tor_{\Z_p} (\log(U_k)/\log(\ov  E_k)) 
= \log(U_k^*)/\log(\ov  E_k)
\end{equation}
be the normalized $p$-adic regulator \cite[\S\,5]{Gra5}.

\smallskip
\noindent
\quad $\bf (d)$ The sub-group of ${\mathcal T}_k$ fixing the 
Bertrandias--Payan field $H_k^\bp$ is isomorphic to~${\mathcal W}_k$
(the field $H_k^\bp$ is the compositum of all $p$-cyclic extensions of $k$ 
embeddable in $p$-cyclic extensions of arbitrary large degree).

\smallskip
Recall some classical fundamental results (under Leopoldt's conjecture) that
may be found in \cite[Corollary III.3.6.3]{Gra1}, \cite[Lemma 3.1, Corollary 3.2]{Gra5},
\cite[D\'efinition 2.11, Proposition 2.12]{J11}, then \cite[\S\,1]{Ng} or \cite{Mov}, via cohomology:

\begin{proposition}
We have the exact sequences:
\begin{equation}
1 \to U_k^*/\ov E_k \too  {\mathcal T}_k 
\too \Gal(k_\infty H_k^\nr/k_\infty) \simeq {\mathcal C}_k \to 1,
\end{equation}
\begin{equation}\label{mu}
1 \to {\mathcal W}_k \too U_k^*/\ov E_k \too  
\log(U_k^*)/\log(\ov  E_k) \simeq {\mathcal R}_k \to 0.
\end{equation}
\end{proposition}

\subsection{Genus theory}\label{normic}

We denote by  $H_{k_n}^\nr$ the $p$-Hilbert class field of $k_n$. 
Since $p$ is totally ramified in $k_n/k$ by convention, the inertia groups 
$I_{\mathfrak p}(k_n/k)$ in $k_n/k$, ${\mathfrak p} \in S$, are isomorphic 
to $G_n = \Gal(k_n/k)$. 

Let $\omega_n$ be the map which associates with 
$\varepsilon \in E_k$ the family of Hasse's 
symbols $\big( \frac{\varepsilon \, ,\, k_n/k}{{\mathfrak p}}\big) \in G_n$, 
${\mathfrak p} \in S$.
This yields the genus exact sequence interpreting the product formula 
of the Hasse symbols \cite[Corollary IV.4.4.1]{Gra1}:
\begin{equation*}
1  \to  E_k/E_k  \cap \No_{k_n/k}(k_n^\times) \mathop{\tooo}^{\omega_n} 
\Omega(k_n/k) \mathop{\tooo}^{\pi_n} \Gal(H_{k_n/k}/ k_n H_k^\nr) \to 1 ,
\end{equation*}

\noindent
where $\Omega(k_n/k)  := \big \{ (\sigma_{\mathfrak p})_{{\mathfrak p} \in S} 
\in G_n^{\order S}, \,\prod_{{\mathfrak p} \in S} \sigma_{\mathfrak p} = 1\big\} 
\simeq G_n^{\order S-1}$, then where $H_{k_n/k}$ is {\it the $p$-genus field} 
of $k_n/k$ defined as the maximal sub-extension of $H_{k_n}^\nr$, abelian over $k$. 
The image of $\omega_n$ is contained in $\Omega(k_n/k)$
and the map $\pi_n$ is defined as follows: 
with $(\sigma_{\mathfrak p})_{{\mathfrak p} \in S} \in G_n^{\order S}$, 
$\pi_n$ associates the product of the extensions $\sigma'_{\mathfrak p}$ of the 
$\sigma_{\mathfrak p}$ in the inertia groups $I_{\mathfrak p}(H_{k_n/k}/ H_k^\nr)$
generating $\Gal(H_{k_n/k}/ H_k^\nr)$; from the product formula, if
$(\sigma_{\mathfrak p})_{{\mathfrak p} \in S} \in \Omega(k_n/k)$, then
$\prod_{{\mathfrak p} \in S} \sigma'_{\mathfrak p}$ fixes both $H_k^\nr$ and $k_n$, 
whence $k_n H_k^\nr$. The genus exact sequence shows that the kernel of 
$\pi_n$ is $\omega_n(E_k)$. 
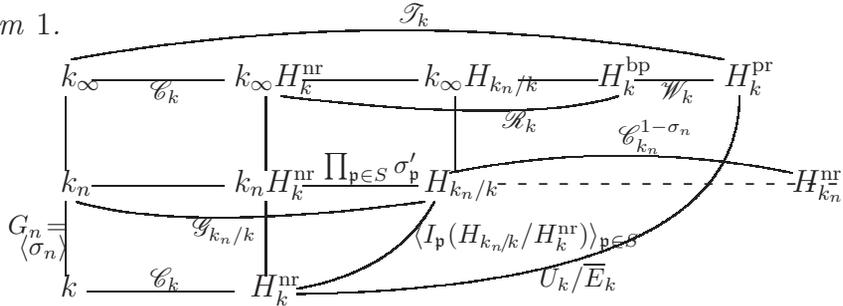
\begin{diagram}\label{diagram1}
\unitlength=0.7cm
$$\vbox{\hbox{\hspace{-6.5cm} \vspace{-0.2cm}
\begin{picture}(11.5,5.2)
\bezier{550}(3.8,4.9)(10.0,6.0)(16.2,4.9)
\put(10.0,5.6){\ft ${\mathcal T}_k$}
\bezier{350}(7.8,4.2)(12.4,3.65)(14.2,4.2)
\put(12.0,3.6){\ft ${\mathcal R}_k$}
\put(5.3,4.15){\ft  ${\mathcal C}_k$}
\put(8.2,4.50){\line(1,0){2.2}}
\put(12.3,4.50){\line(1,0){1.5}}
\put(14.5,4.50){\line(1,0){1.5}}
\put(13.8,4.4){$H_k^\bp$}
\put(15.0,4.15){\ft ${\mathcal W}_k$}
\put(4.2,4.50){\line(1,0){2.5}}
\put(6.9,4.4){$k_\infty H_k^\nr$}
\put(3.6,4.4){$k_\infty$}
\put(10.5,4.4){$k_\infty H_{k_n/k}$}
\put(16.2,4.4){$H_k^\pr$}
\put(8.2,2.50){\line(1,0){2.2}}
\put(11.9,2.41){- - - - - - - - - - - - - - - - -}
\put(4.2,2.50){\line(1,0){2.5}}
\put(4.1,0.50){\line(1,0){2.8}}
\put(8.6,2.75){\ft $\prod_{{\mathfrak p} \in S} \sigma'_{\mathfrak p}$}
\put(14.2,3.3){\ft ${\mathcal C}_{k_n}^{1-\sigma_n}$}
\bezier{250}(11.0,2.75)(14.25,3.4)(17.5,2.75)
\put(5.3,0.6){\ft  ${\mathcal C}_k$}
\put(3.7,0.8){\line(0,1){1.4}}
\put(7.5,0.8){\line(0,1){1.4}}
\put(3.7,2.8){\line(0,1){1.4}}
\put(7.5,2.8){\line(0,1){1.4}}
\put(11.1,2.8){\line(0,1){1.4}}
\bezier{450}(3.9,2.2)(6.0,1.6)(10.5,2.2)
\put(6.1,1.55){\ft  ${\mathcal G}_{k_n/k}$}
\put(10.5,2.4){$H_{k_n/k}$}
\put(17.5,2.4){$H_{k_n}^\nr$}
\put(6.9,2.4){$k_n H_k^\nr$}
\put(3.6,2.4){$k_n$}
\put(7.2,0.4){$H_k^\nr$}
\put(3.6,0.4){$k$}
\put(2.6,1.6){\ft $G_n \!=$}
\put(2.8,1.2){\ft $\langle \sigma_n \rangle$}
\bezier{350}(8.1,0.55)(10.0,0.9)(10.7,2.2)
\put(10.3,1.42){\ft $\langle I_{\mathfrak p}
(H_{k_n\!/\!k}/ H_k^\nr) \rangle_{{\mathfrak p} \in S}$}
\bezier{550}(8.1,0.45)(16.5,0.6)(16.5,4.2)
\put(12.7,0.6){\ft $U_k/\ov E_k$}
\end{picture}   }} $$
\end{diagram}
\unitlength=1.0cm
We have, using Chevalley's ambiguous class number 
formula \cite[p.\,402]{Che}:
\begin{equation}\label{chevalley}
\order {\mathcal G}_{k_n/k} =\order \Gal(H_{k_n/k}/k_n)
=  \ffrac{\order {\mathcal C}_{k_n}}{\order {\mathcal C}_{k_n}^{1-\sigma_n}}
= \order  {\mathcal C}_k \cdot \ffrac{p^{n \cdot (\order S -1)}}
{(E_k : E_k \cap \No_{k_n/k}(k_n^\times))}
\end{equation}

In the Diagram, the genus field
$H_{k_n/k}$ is the fixed field of the image of ${\mathcal C}_{k_n}^{1-\sigma_n}$,
where ${\mathcal G}_{k_n/k} = \Gal(H_{k_n/k}/ k_n)$ is the genus group in $k_n/k$.

\subsection{Groups ${\mathcal R}_k^\nr$, ${\mathcal R}_k^\ram$  -- 
Ramification in $H_k^\pr/k_\infty$}\label{ram}

The genus group ${\mathcal G}_{k_n/k}$ has, in our context, the following 
main property that will give Theorem \ref{genus2} when $n$ is large enough:

\begin{lemma} \label{genus} 
For all $n \geq 0$, $k_\infty H_{k_n/k} \subseteq H_k^\bp$. Then
$\order {\mathcal G}_{k_n/k} \big\vert \,\order {\mathcal C}_k \cdot  {\mathcal R}_k$,
which is equivalent (using formula \eqref{chevalley}) to $\ffrac{p^{n \cdot (\order S -1)}}
{(E_k : E_k \cap \No_{k_n/k}(k_n^\times))} \,\big\vert \,\order {\mathcal R}_k$.
\end{lemma}

\begin{proof}
Indeed, using the idelic global reciprocity map (under Leopoldt's conjecture), 
we have the fundamental diagram \cite[\S\,III.4.4.1]{Gra1} of the Galois 
group of the maximal abelian pro-$p$-extension $k^\ab$ of $k$, 
with our present notations, where $F_v$ is the residue field of the 
tame place $v$ (finite or infinite) and where  $H_k^\ta$ is
the maximal tame sub-extension of $k^\ab$.
The fixed field of $U_k = \bigoplus_{{\mathfrak p} \in S} U_{\mathfrak p}$ 
is $H_k^\ta$ since each $U_{\mathfrak p}$ is the inertia group of ${\mathfrak p}$ 
in $k^\ab/k$. Thus, $\tor_{\Z_p} (U_{\mathfrak p}) = \mu_{\mathfrak p}$, 
restricted to $\Gal(H_k^\pr/k)$, fixes $k_\infty$ and since 
$k_\infty H_{k_n/k}/k_\infty$ is unramified, it fixes 
$k_\infty H_{k_n/k}$ for all $n \geq 0$. 
\begin{diagram}\label{diagram2}
\unitlength=0.6cm
$$\vbox{\hbox{\hspace{-1.0cm} 
\begin{picture}(11.5,3.3)
\put(8.5,2.50){\line(1,0){3.0}}
\put(1.6,2.50){\line(1,0){5.9}}
\put(1.6,0.50){\line(1,0){5.9}}
\put(1.0,0.9){\line(0,1){1.20}}
\put(0.0,0.5){\line(1,0){0.70}}
\put(8.00,0.9){\line(0,1){1.20}}
\bezier{400}(1.2,2.85)(6.45,3.4)(11.75,2.85)
\put(5.4,3.4){\ft ${\prod_{v \notin S}{F_v^\times \! \otimes\! \Z_p}}$}
\bezier{280}(8.5,0.45)(11.5,0.5)(11.9,2.2)
\put(11.2,0.8){\ft $U_k \!=\! \bigoplus_{{\mathfrak p} \in S} U_{\mathfrak p}$}
\bezier{250}(8.5,2.3)(10.0,1.8)(11.6,2.3)
\put(9.4,1.6){\ft $E_k \! \otimes\! \Z_p$}
\put(11.7,2.3){$k^\ab$}
\put(7.6,2.3){$M_0$}
\put(0.6,2.3){$H_k^\pr$}
\put(7.7,0.4){$H_k^\ta$}
\put(0.7,0.4){$H_k^\nr$}
\put(-0.4,0.4){$k$}
\put(1.2,1.4){\ft $U_k/\ov E_k$}
\end{picture}   }} $$
\end{diagram}
\unitlength=1.0cm
\noindent
In Diagram \ref{diagram1}, the restriction of 
$W_k = \bigoplus_{{\mathfrak p} \in S}\mu_{\mathfrak p}$
to $\Gal(H_k^\pr/H_k^\nr)$ is isomorphic to $W_k/\mu_k = {\mathcal W}_k$
whose fixed field is $H_k^\bp$; whence the first claim. 

\smallskip
The second one is obvious since non-ramification propagates. 
Then $\order {\mathcal G}_{k_n/k}$ increases with $n$ and 
stabilizes at a divisor of 
$[H_k^\bp : k_\infty] = \order {\mathcal C}_k \!\cdot \order {\mathcal R}_k$. 
\end{proof}

Put ${\mathcal G}_k \simeq {\mathcal G}_{k_n/k}$ for $n$ large enough. 
This group is called the genus group of $k_\infty/k$; then the field $H_k^\gen 
:= \hbox{$\bigcup_m$} \, H_{k_m/k}$ (the genus field of $k_\infty/k$) is unramified 
over $k_\infty$ of Galois group ${\mathcal G}_k$. We can state more precisely:

\begin{theorem}\label{defhgen}
Let $n_0\geq 0$ be such that $\order {\mathcal G}_{k_{n_0}/k}$ stabilizes,
definig the genus field $H_k^\gen$ such that 
$\Gal(H_k^\gen/k_\infty) = {\mathcal G}_k$. 
Then $H_k^\gen$ is the maximal unramified exten\-sion of 
$k_\infty$ in $H_k^\pr$ and $\Gal(H_k^\pr/H_k^\gen)$
$\simeq \big \langle\, \tor_{\Z_p}(U_{\mathfrak p} \ov E_k/ \ov E_k)
\,\big \rangle_{{\mathfrak p} \in S}$. 
\end{theorem}

\begin{proof}
To simplify, put $L_\infty := H_k^\gen$.
Let $L'_\infty$ be a degree $p$ unramified extension of $L_\infty$ 
in $H_k^\bp$; put $L=H_{k_n/k}$, $n \geq n_0$, and consider 
$L'$ such that $L' \cap L_\infty = L$ and 
$L' L_\infty = L'_\infty$; thus $\Gal(L_\infty/L) \simeq 
\Gal(L'_\infty/L') \simeq \Z_p$. 
Taking $n \gg n_0$, one may assume that
$L_\infty/L$ and $L'_\infty/L'$ are totally ramified at $p$.

\smallskip
Let $M \ne L'$ be a degree $p$ extension of $L$ 
in $L'_\infty$ and $v$ a $p$-place of $L$; if $v$ was 
unramified in $M/L$, the non-ramification would propagate 
over $L'$ in $L'_\infty$ (a contradiction). 
Thus, the inertia group of $v$ in $L'_\infty/L$ 
is necessarily  $\Gal(L'_\infty/L)$ or $\Gal(L'_\infty/L')$,
but this last case for all $v$ gives $L'/L/k_n$ unramified 
and $L'/k$ abelian (absurd by definition of the genus field 
$L=H_{k_n/k}$); so there exists $v_0$ totally ramified in 
$L'_\infty/L$, hence in $L'_\infty/L_\infty$ (absurd).

\smallskip
For ${\mathfrak p} \in S$, the inertia group
$I_{\mathfrak p}(H_k^\pr/k_\infty)$ is isomorphic to the torsion part,
$\tor_{\Z_p}(U_{\mathfrak p} \ov E_k/ \ov E_k)$, of the image
of $U_{\mathfrak p}$ in $U_k/ \ov E_k$.
\end{proof}

Let ${\mathcal R}_k^\nr := \Gal(H_k^\gen / k_\infty H_k^\nr)$,
${\mathcal R}_k^\ram := \Gal(H_k^\bp/H_k^\gen)$.
The top of Diagram \ref{diagram1} may be specified as follows
(with $H_k^\pr/H_k^\gen$ totally ramified at $p$):
\begin{diagram}\label{diagram3}
\unitlength=0.8cm
$$\vbox{\hbox{\hspace{-2.7cm} \vspace{0.4cm}
\begin{picture}(11.5,1.2)
\bezier{550}(0.8,0.8)(7.0,2.0)(13.2,0.8)
\put(7.0,1.5){\ft  ${\mathcal T}_k$}
\bezier{350}(4.4,0.2)(7.9,-0.3)(11.1,0.2)
\bezier{400}(4.3,0.2)(8.6,-1.2)(13.2,0.2)
\put(11.0,-0.6){\ft $U_k^*/\ov E_k$}
\put(8.0,-0.35){\ft  ${\mathcal R}_k$}
\put(5.8,0.7){\ft  ${\mathcal R}_k^\nr$}
\put(9.0,0.7){\ft  ${\mathcal R}_k^\ram$}
\put(2.4,0.15){\ft  ${\mathcal C}_k$}
\put(5.1,0.50){\line(1,0){2.0}}
\put(7.15,0.4){$H_k^\gen$}
\put(7.8,0.50){\line(1,0){3.0}}
\put(11.5,0.50){\line(1,0){1.5}}
\put(10.85,0.4){$H_k^\bp$}
\put(12.0,0.15){\ft ${\mathcal W}_k$}
\put(1.2,0.50){\line(1,0){2.5}}
\put(3.8,0.4){$k_\infty H_k^\nr$}
\put(0.6,0.4){$k_\infty$}
\put(13.1,0.4){$H_k^\pr$}
\bezier{350}(0.9,0.2)(4.0,-0.7)(7.1,0.2)
\put(3.7,-0.55){\ft  ${\mathcal G}_k$}
\end{picture}   }} $$
\end{diagram}
\unitlength=1.0cm

From Lemma \ref{genus}, formula \eqref{chevalley} and the above study, we can state
(a generalization of Taya analytic viewpoint \cite[Theorem 1.1]{T1}):
\begin{theorem} \label{genus2} 
Let $n \gg 0$ be such that ${\mathcal G}_{k_n/k} :=
\Gal(H_{k_n/k} / k_n) \simeq {\mathcal G}_k$. Then 
$\order {\mathcal G}_k = \order {\mathcal C}_k \cdot \order {\mathcal R}_k^\nr =
{\mathcal C}_{k_n}^{G_n}$, equivalent to $\ffrac{p^{n \cdot (\order S -1)}}
{(E_k : E_k \cap \No_{k_n/k}(k_n^\times))} = \order {\mathcal R}_k^\nr$.
\end{theorem}

\section{Filtration of ${\mathcal C}_{k_n}$ -- Class and Norm factors}\label{CNfactors}
Describe now a formal algorithm of computation of $\order {\mathcal C}_{k_n}$, 
for all $n \geq 0$, by means of ``unscrewing'' in $k_n/k$. 
For this, put $G_n := \Gal(k_n/k) =: \langle\, \sigma_n\, \rangle$.
Let $I_{k_n}$ be the group of prime-to-$p$ ideals of $k_n$.

\subsection{Filtration of the class groups}
One uses the filtration of $M_n := {\mathcal C}_{k_n}$ defined as follows \cite[Corollary 3.7]{Gra2}.
For $n \geq 0$ fixed, $(M_n^i)_{i \geq 0}$ is the $i$-sequence of sub-$G_n$-modules of $M_n$ 
defined by $M_n^0:= 1$ and $M_n^{i+1}/M_n^i := (M_n/M_n^i)^{G_n}$, for $0\leq i \leq b_n$,
where $b_n$ is the least integer $i$ such that $M_n^i = M_n$ 
(i.e., such that $M_n^{i+1} = M_n^i$). 

\smallskip
If ${\mathcal C}_k=1$, $M_0=M_0^0=1$, $b_0=0$; 
if ${\mathcal C}_k \ne 1$, $M_0=M_0^1={\mathcal C}_k$, $b_0=1$. 

\smallskip
\noindent
We will obtain, inductively, ideal groups $J_n^i \subset I_{k_n}$, 
with $J_n^0=1$, such that:
$$\hbox{$M_n^i =: \cl_{k_n}(J_n^i)$, for all $i \geq 0$}. $$

\begin{proposition} \label{filtration}
This filtration has the following properties:

\noindent
\quad (i) From $M_n^0=1$, one gets $M_n^1 = M_n^{G_n}$ of order 
$\order {\mathcal C}_k \cdot \ffrac{p^{n \cdot (\order S - 1)}}
{(E_k : E_k \cap \No_{k_n/k}(k_n^\times))}$. 

\noindent
\quad (ii) One has
$M_n^i= \{c \in M_n, \, c^{(1-\sigma_n)^i} =1 \}$, for all $i \geq 0$.

\smallskip
\noindent
\quad (iii) The $i$-sequence $\order (M_n^{i+1}/M_n^i)$, 
$0 \leq i \leq b_n$, is {\it decreasing} to $1$ and is bounded by
$\order  M_n^1$ since $1-\sigma_n$ defines the
injections $M_n^{i+1}/M_n^i \!\hookrightarrow\! M_n^i/M_n^{i-1}$.

\smallskip
\noindent
\quad (iv) $\order M_n = \order M_n^{b_n} = \prod_{i=0}^{b_n-1} \order (M_n^{i+1}/M_n^i)$.
\end{proposition}

In \cite[Formula (29), \S\,3.2]{Gra2}, we established a
generalization of Chevalley's ambiguous class number formula,
by means of the norm groups
$\No_{k_n/k}(M_n^i) = \cl_k (\No_{k_n/k}(J_n^i))$ and the subgroups
$\Lambda_n^i := \{x \in k^\times,\, (x) \in \No_{k_n/k}(J_n^i)\}$ of $k^\times$, 
giving $\order  \big( M_n^{i+1} / M_n^i \big)= 
\ffrac{\order  {\mathcal C}_k}{\order \No_{k_n/k}( M_n^i)} \cdot  
\ffrac{p^{n \cdot (\order S -1)}}
{(\Lambda_n^i : \Lambda_n^i \cap \No_{k_n/k}(k_n^\times))}$, where:
\begin{equation}\label{cr}
\frac{\order  {\mathcal C}_k}{\order \No_{k_n/k}( M_n^i)}
\hspace{0.5cm} \& \hspace{0.5cm}
\frac{p^{n \cdot (\order S -1)}}{(\Lambda_n^i :  \Lambda_n^i \cap 
\No_{k_n/k}(k_n^\times))}
\end{equation}

\noindent
are integers called {\it the class factor} and {\it the norm factor}, respectively,
at the step $i$ of the algorithm in the layer $k_n$. These factors are independent 
of the choice of the ideals defining $J_n^i$ up to principal ideals of $k_n$ and the groups
$\Lambda_n^i$ are, therefore, defined up to elements of $\No_{k_n/k}(k_n^\times)$.

\smallskip
From Lemma \ref{genus} and Diagram \ref{diagram3},
we can state, for any fixed integer $n$ and for the class and norm factors \eqref{cr}:

\begin{corollary}\label{fact}
The class factors divide $\order {\mathcal C}_k$ and define a decreasing
$i$-sequence since $\No_{k_n/k}( M_n^i) \subseteq \No_{k_n/k}( M_n^{i+1})$
for all $i \geq 0$. 
The norm factors divide $\order {\mathcal R}_k^\nr$ and define a decreasing 
$i$-sequence for all $i \geq 0$, due to the injective maps 
$E_k/E_k \cap \No_{k_n/k}(k_n^\times) \! \hookrightarrow \cdots
\Lambda_n^i/\Lambda_n^i \!\cap\!  \No_{k_n/k}(k_n^\times) 
\! \hookrightarrow \Lambda_n^{i+1}/\Lambda_n^{i+1} \!\cap\! 
\No_{k_n/k}(k_n^\times) \cdots$
\end{corollary}

\subsection{Relation of the algorithms with Iwasawa's theory}
The subgroups $J_n^i$ of $I_{k_n}$ are built inductively from 
$J_n^0 = 1$, hence $\Lambda_n^0 = E_k$. More precisely the 
algorithm is the following, for $n$ and $i$ fixed \cite[\S\,6.2]{Gra3}:

\smallskip
Let $x \in \Lambda_n^i$, $(x) = \No_{k_n/k}({\mathfrak A})$, 
${\mathfrak A} \in J_n^i$; thus $x$ is local norm on the tame places.
Suppose that $x$ is local norm on $S$, hence 
global norm and we can write $x = \No_{k_n/k}(y)$, $y \in k_n^\times$. The random 
aspects occur, from the relation $\No_{k_n/k}(y) = \No_{k_n/k}({\mathfrak A})$, 
in the mysterious ``evolution relation'' giving the existence of an ideal 
${\mathfrak B} \in I_{k_n}$ 
such that $(y) = {\mathfrak A}\,{\mathfrak B}^{1-\sigma_n}$. Remark that for
$\No(y)=1$ and $y = b^{1-\sigma_n}$, $b$ is given by an {\it additive Hilbert's resolvent}.

\smallskip
A priori there is no algebraic link with the previous data because of the 
global solution $y$ (Hasse's norm theorem) unique up to
$k_n^\times{}^{1-\sigma_n}$; this gives ${\mathfrak B}$ up to principal ideals.
All numbers $x \in \Lambda_n^i \cap \No_{k_n/k}(k_n^\times)$ define the step $i+1$:
$$J_n^{i+1} := J_n^i \cdot \langle \ldots, {\mathfrak B}, \ldots \rangle
\ \hbox{and}\  \Lambda_n^{i+1} := \{x \in k^\times,\, (x) \in \No_{k_n/k}(J_n^{i+1})\}. $$

Therefore, for $i=b_n$ we obtain
$M_n^{b_n} = {\mathcal C}_{k_n}$, $\No_{k_n/k} \big (M_n^{b_n} \big) = {\mathcal C}_k$ and
$(\Lambda_n^{b_n}  : \Lambda_n^{b_n} \cap \No_{k_n/k}(k_n^\times)) 
= p^{n \cdot (\order S -1)}$, which explains that $\order  {\mathcal C}_{k_n}$ essentially
depends on the number of steps $b_n$ of the algorithm; this is expressed in terms 
of Iwasawa invariants as follows:

\begin{theorem} \label{O} 
We assume the Conventions \ref{conventions} for the base field $k$ and recall that
${\mathcal R}_k^\nr :=\Gal(H_k^\gen/k_\infty H_k^\nr)$ (Diagram \ref{diagram3}), 
where $H_k^\gen$ is the genus field of $k_\infty/k$ (Theorem \ref{defhgen}).
Let $b_n$ be the length of the algorithm in the layer $k_n$. Then (where $v_p$ 
denotes the $p$-adic valuation):

\smallskip
\noindent
\quad (i) $b_n \leq \lambda \cdot n + \mu \cdot p^n + \nu \leq 
v_p(\order {\mathcal C}_k \cdot \order {\mathcal R}_k^\nr) \cdot b_n$, for all $n \geq 0$.
So,  $\lambda=\mu= 0$ $\Longleftrightarrow$ $b_n$ bounded.

\smallskip
\noindent
\quad (ii) $b_ m \geq b_n$, for all $m \geq n \geq 0$.

\smallskip
\noindent
\quad (iii) $b_1=0$ $\Longleftrightarrow$ $\lambda=\mu=\nu= 0$
$\Longleftrightarrow$  $b_n=0$ for all $n$ $\Longleftrightarrow$
${\mathcal C}_k = {\mathcal R}_k^\nr =1$.
\end{theorem}

\begin{proof} Let $M_n := {\mathcal C}_{k_n}$, for all $n \geq 0$.

\smallskip
(i) As $\order \big( M_n^{i+1}/M_n^i\big) \geq p$, for $0 \leq i \leq b_n-1$,
Proposition \ref{filtration}\,(iv) implies
$\order M_n = \order M_n^{b_n} \geq p^{b_n}$; whence
$b_n \leq \lambda \cdot n + \mu \cdot p^n +  \nu$.

\smallskip
From the fact that
$\order \big( M_n^{i+1}/M_n^i \big) \mid \order {\mathcal C}_k \cdot \order {\mathcal R}_k^\nr$
(Corollary \ref{fact}) this yields
$\order  \big( M_n^{i+1}/M_n^i \big) \leq \order {\mathcal C}_k \cdot \order {\mathcal R}_k^\nr$ 
for $0 \leq i\leq b_n-1$, whence $\order {\mathcal C}_{k_n} \leq 
(\order {\mathcal C}_k \cdot \order {\mathcal R}_k^\nr)^{b_n}$ 
from Proposition~\ref{filtration}\,(iv); hence the second inequality and the second claim.

\smallskip
(ii) By definition, $M_m^{b_m} = M_m$ with $b_m$ minimal.
Since $k_m/k_n$ is totally ramified, $\No_{k_m/k_n}(M_m^{b_m})=M_n$,
but $\No_{k_m/k_n}(M_m^{b_m}) \subseteq M_n^{b_m}$ (Proposition \ref{filtration}\,(ii)),
whence $M_n \subseteq M_n^{b_m}$, thus $M_n^{b_m} = M_n$, proving the claim.

\smallskip
(iii)  So $b_1=0$ implies $b_0=0$, whence $\lambda+\mu p + \nu = \mu+ \nu=0$
yielding $\lambda=\mu=0$ and $\nu=0$; then (i) implies $b_n=0$
for all $n \geq 0$, in other words, ${\mathcal C}_{k_n}=1$ for 
all $n \geq 0$; thus, taking $n \gg 0$ to apply Theorem \ref{genus2} yields 
${\mathcal G}_{k_n/k} = {\mathcal C}_{k_n}^{G_n} = 1$, whence 
${\mathcal C}_k = {\mathcal R}_k^\nr =1$ (reciprocals obvious).
\end{proof}

\begin{corollary}\label{bn=1}
(a) Under Conventions \ref{conventions},
$\lambda = \mu =0$ (equivalent to $b_n$ bounded) is
equivalent to each of the following properties:

\smallskip
\quad (i) $\No_{k_n/k} : {\mathcal C}_{k_n} \to {\mathcal C}_k$ 
is an isomorphisms for all $n \geq 0$.

\smallskip
\quad (ii) $\order {\mathcal C}_{k_n}^{G_n} = \order {\mathcal C}_{k_n} 
= \order {\mathcal C}_k$, for all $n \geq 0$.

\smallskip
\quad (iii) ${\mathcal C}_{k_n}^{G_n} = {\mathcal C}_{k_n}$, for all $n \geq 0$ 
and ${\mathcal R}_k^\nr =1$.

\smallskip
(b) Let $k_{n_1}$, still denoted $k$, be such that $b_n$ is constant for all
$n \geq n_1$;\,\footnote{\,On must note that for each change of base field
in the tower, the Iwasawa invariants are given by Conventions \ref{conventions}, 
and the algorithms are distinct; for instance the parameter $b_n$ defines a new 
function of the $n$th layer of the new $k$ (in the meaning $[k_n : k] = p^n$).} 
for this new base field $k$ and the new $b$-function, $b_n \leq 1$, for all $n \geq 0$. 
\end{corollary}

\begin{proof}
Proof of (a).
(i) Under the condition $\lambda = \mu = 0$, 
$\order {\mathcal C}_{k_n} = \order {\mathcal C}_k = p^\nu$ 
for all $n$, and all the (surjective) norm maps are isomorphisms.

(ii) Chevalley's formula
$\order {\mathcal C}_{k_n}^{G_n} = \order {\mathcal C}_k \cdot 
\frac{p^{n \cdot (\order S -1)}}{(E_k : E_k \cap \No_{k_n/k}(k_n^\times))} 
\leq \order {\mathcal C}_{k_n} = \order {\mathcal C}_k$ yields 
$\order {\mathcal C}_{k_n}^{G_n} = \order {\mathcal C}_{k_n} = \order  {\mathcal C}_k$ 
and $\frac{p^{n \cdot (\order S -1)}}{(E_k : E_k \cap \No_{k_n/k}(k_n^\times))}=1$,
for all $n \geq 0$. 

(iii) From (ii), ${\mathcal R}_k^\nr =1$, taking $n \gg 0$ to apply Theorem \ref{genus2}.

\smallskip
\noindent
In the three cases, the reciprocals are obvious.

\smallskip
Proof of (b). 
Consider the second step of the algorithm in $k_n$ 
(we exclude the case $b_n=0$ where all class 
groups are trivial); the class factor for ${\mathcal C}_{k_n}^2/{\mathcal C}_{k_n}^1$ 
is trivial since $\No_{k_n/k} ({\mathcal C}_{k_n}^{G_n}) = {\mathcal C}_k$
(from (i), (ii)) and the norm factor, as divisor of ${\mathcal R}_k^\nr$, is also trivial
(from (iii)); whence $b_n = 1$ for all $n \geq 0$. 
\end{proof}

Note that under Greenberg's conjecture, in $\frac{p^{n \cdot (\order S -1)}}
{(\Lambda_n^1 : \Lambda_n^1\cap \No_{k_n/k}(k_n^\times))}$, we have
$\Lambda_n^1 = \{x \in k^\times, \ (x) = \No_{k_n/k}({\mathfrak A}),\, 
{\mathfrak A} \in J_n^1\}$ where $\cl_{k_n}(J_n^1) = {\mathcal C}_{k_n}^{G_n}$;
thus, norms being isomorphisms, $(x) = \No_{k_n/k}({\mathfrak A})$ implies
that ${\mathfrak A}=(\alpha)$, $\alpha \in k_n^\times$, so that
$\Lambda_n^1=E_k\,\No_{k_n/k}(k_n^\times)$, showing that the 
algorithm becomes trivial.

\subsection{The $n$-sequences $({\mathcal C}_{k_n} / {\mathcal C}_{k_n}^i)^{G_n}$}

We fix the step $i$ of the algorithms. For now, we do not assume the 
Conventions \ref{conventions}. For all $m \geq n \geq 0$, the norm 
maps $\No_{k_m/k_n}$ on $M_m$ and $M_m^{(1-\sigma_m)^i}$ are surjective
(they are, a priori, not injective nor surjective on the kernels $M_m^i$ of
the maps $M_m \to M_m^{(1-\sigma_m)^i}$).
This leads to the following result (see \cite[Lemmas 7.1, 7.2]{Gra3}
for the details), giving another approach of the conjecture:

\begin{theorem}\label{thm67}
For all $i \geq 0$ fixed, $\big\{ \order \big( M_n^{i+1} / M_n^i \big) \big\}_n$ 
defines an increasing $n$-sequence of divisors of
$\order {\mathcal C}_k \cdot \order {\mathcal R}_k^\nr$.
Thus $\ds \lim_{n \to \infty} \order  \big( M_n^{i+1} / M_n^i \big) =:
p^{c^i}  p^{\rho^i}$.
The $i$-sequences $p^{c^i}$ and $ p^{\rho^i}$ are decreasing, 
stationary at a divisor $p^c$ of $\order {\mathcal C}_k$ and
$p^\rho$ of $\order {\mathcal R}_k^\nr$, respectively. 
Greenberg's conjecture is equivalent to $c = \rho =0$.
\end{theorem}

\section{${\mathcal T}_k$ as governing invariant of the algorithms}\label{sec4}

The ideals ${\mathfrak A} \in J_n^i$
may be arbitrarily modified up to principal ideals of $k_n$, whence
$\No_{k_n/k}({\mathfrak A})$ defined up to elements of
$\No_{k_n/k}(k_n^\times)$, as well as $ \Lambda_n^i$. 
We intend to obtain suitable {\it finite sets} of representatives of 
these ideal norms, independently of $n$, more precisely of cardinality 
$\leq \order {\mathcal T}_k$. 

\subsection{Decomposition of $\No_{k_n/k}({\mathfrak A})$ -- 
The fundamental ideals ${\mathfrak t}$}

Let $H_k^\pr$ and $H_{k_n}^\pr$ be the maximal abelian $p$-ramified 
pro-$p$-extensions of $k$ and $k_n$, respectively. 
Let $F$ be an extension of $H_k^\nr$
such that $H_k^\pr$ be the direct compositum 
of $F$ and $k_\infty H_k^\nr$ over $H_k^\nr$ (possible 
because $k_\infty \cap H_k^\nr = k$ due to the total ramification of $p$
in $k_\infty/k$); we~put $\Gamma = \Gal(H_k^\pr/F) \simeq \Z_p$. 

\smallskip
In the same way, we fix an extension $F_n$ 
of $F$ such that $H_{k_n}^\pr$ be the direct compositum 
of $F_n$ and $H_k^\pr$ over $k_n F$; we put $\Gamma_n
= \Gal(H_{k_n}^\pr/F_n) \simeq \Gamma^{p^n}$. We have
$F=F_0 \subset F_1 \subset \cdots \subset F_n \subset F_{n+1}\subset \cdots$
(see Diagram \ref{diagram4} hereafter).

\smallskip
In what follows, we systematically use the flatness of $\Z_p$.

Consider the Artin symbols $\Big( \frac{H_k^\pr/k}{\cdot} \Big)$ and 
$\Big( \frac{H_{k_n}^\pr/k_n}{\cdot} \Big)$,
defined on $I_k \otimes \Z_p$ and $I_{k_n} \otimes \Z_p$, respectively.
Their images are the Galois groups ${\mathcal A}_k$ 
(resp. ${\mathcal A}_{k_n}$); their kernels are
the groups of infinitesimal principal ideals
${\mathcal P}_{k, \infty}$ (resp. ${\mathcal P}_{k_n, \infty}$), 
where ${\mathcal P}_{k, \infty}$ is the set of ideals 
$(x_\infty)$, $x_\infty \in k'^\times \otimes \Z_p$, such that 
$\iota x_\infty = 1$ in $U_k$ (idem for 
${\mathcal P}_{k_n, \infty}$) \cite[Theorem III.2.4, Proposition III.2.4.1]{Gra1}. 

\smallskip
The arithmetic norm (or restriction of automorphisms), in $k_n/k$, 
leads to $\No_{k_n/k}({\mathcal A}_{k_n}) = \Gal(H_k^\pr/k_n)$ and
$\No_{k_n/k}({\mathcal T}_{k_n}) = {\mathcal T}_k$ 
since $k_n{}_\infty=k_\infty$.
The fixed points formula ${\mathcal T}_{k_n}^{G_n} \simeq {\mathcal T}_k$
(\cite[Theorem IV.3.3]{Gra1}, \cite[Section 2 (c)]{J0}), implies
$\Ker(\No_{k_n/k}) = {\mathcal T}_{k_n}^{1-\sigma_n} = \Gal(H_{k_n}^\pr/H_k^\pr)$.

\smallskip
We denote by ${\mathcal K}^\times_\infty \subset k'^\times \otimes \Z_p$ the 
subgroup of infinitesimal elements of $k$ 
(idem for ${\mathcal K}^\times_{n,\infty} \subset  k_n'^\times \otimes \Z_p$).
In the sequel, the notations $x_\infty$, $y_\infty$,\,$\ldots$ always denote
such infinitesimal elements. 
\begin{diagram}\label{diagram4}
\unitlength=0.91cm 
$$\vbox{\hbox{\hspace{-4.0cm}\vspace{-0.2cm}
\begin{picture}(11.5,5.6)
\put(6.3,4.50){\line(1,0){1.4}}
\put(6.35,2.50){\line(1,0){1.1}}
\put(6.0,0.45){\line(1,0){1.7}}
\put(8.4,4.50){\line(1,0){0.6}}
\put(9.05,4.4){$H_k^\pr$}
\put(9.5,4.50){\line(1,0){1.3}}
\put(3.7,4.50){\line(1,0){1.5}}
\put(3.7,2.50){\line(1,0){1.6}}
\put(3.7,0.45){\line(1,0){1.6}}
\put(8.5,4.25){\ft ${\mathcal W}_k$}
\bezier{450}(3.6,4.8)(7.5,6.6)(10.8,4.8)
\put(6.8,5.8){\ft $ {\mathcal T}_{k_n} = 
\tor_{\Z_p}({\mathcal A}_{k_n})$}
\bezier{350}(3.8,4.65)(6.35,5.55)(8.95,4.65)
\put(6.0,5.2){\ft ${\mathcal T}_k = 
\tor_{\Z_p}({\mathcal A}_k)$}
\put(8.4,0.45){\line(1,0){0.6}}
\put(9.1,0.4){$F$}
\put(3.50,2.7){\line(0,1){1.55}}
\put(3.50,0.7){\line(0,1){1.55}}
\put(5.7,2.7){\line(0,1){1.55}}
\put(5.7,0.7){\line(0,1){1.55}}
\put(8.0,0.7){\line(0,1){1.55}}
\put(8.0,2.7){\line(0,1){1.55}}
\put(9.28,0.7){\line(0,1){1.55}}
\put(9.28,2.7){\line(0,1){1.6}}
\put(9.1,2.4){$k_nF$}
\put(8.4,2.5){\line(1,0){0.6}}
\put(11.28,2.7){\line(0,1){1.6}}
\put(11.15,2.4){$F_n$}
\put(9.8,2.5){\line(1,0){1.2}}
\put(4.4,4.25){\ft ${\mathcal C}_k$}
\put(6.8,4.25){\ft ${\mathcal R}_k$}
\bezier{400}(3.8,2.56)(8.0,3.0)(11.0,4.3)
\put(9.4,3.55){\ft ${\mathcal A}_{k_n}$}%
\bezier{400}(3.8,0.55)(6.8,0.8)(9.2,4.3)
\put(6.7,1.5){\ft ${\mathcal A}_k$}
\bezier{400}(3.8,2.6)(8.4,3.5)(9.15,4.3)
\put(6.5,3.55){\ft $\No{\mathcal A}_{k_n}$}%
\put(10.85,4.4){$H_{k_n}^\pr$}
\put(5.2,4.4){$k_\infty H_k^\nr$}
\put(7.8,4.4){$H_k^\bp$}
\put(3.3,4.4){$k_\infty$}
\put(5.3,2.4){$k_n H_k^\nr$}
\put(7.55,2.4){$k_nF^\bp$}
\put(5.35,0.4){$H_k^\nr$}
\put(7.8,0.4){$F^\bp$}
\put(3.4,2.4){$k_n$}
\put(3.4,0.40){$k$}
\put(2.65,1.5){$G_n$}
\bezier{200}(3.3,0.7)(3.0,1.4)(3.3,2.3)
\put(3.6,1.5){\ft  $p^n$}
\bezier{400}(11.6,2.5)(12.2,3.45)(11.5,4.4)
\put(12.0,3.35){$\Gamma_n$}
\bezier{400}(5.95,0.55)(8.25,1.5)(9.25,4.3)
\put(6.8,0.8){\ft $U_k\!/\!\ov E_k$}
\bezier{400}(9.6,0.5)(11.3,2.55)(9.6,4.4)
\put(10.5,2.7){$\Gamma$}
\end{picture} }} $$
\end{diagram}
\unitlength=1.0cm

\begin{lemma}\label{Ninfty} 
If $(x_\infty) \in {\mathcal P}_{k, \infty} \cap \No_{k_n/k}(I_{k_n} \otimes \Z_p)$,
then $x_\infty \in \No_{k_n/k}({\mathcal K}^\times_{n,\infty})$.
\end{lemma}

\begin{proof} 
The assumption implies that $x_\infty$ is everywhere local 
norm in $k_n/k$, whence $x_\infty = \No_{k_n/k}(y)$, 
$y \in k_n'^\times \otimes \Z_p$ (Hasse norm theorem). 
Thus, we get $\iota \No_{k_n/k}(y) = 
\No_{k_n/k}(\iota_n y) = 1$ and $\iota_n y = t^{1-\sigma_n}$, 
$t \in \prod_{{\mathfrak p}_n\in S_n}  k^\times_{n,{\mathfrak p}_n}$
(Hilbert's Theorem $90$, $\Hom^1(G_n,\prod_{{\mathfrak p}_n\in S_n} 
\!\! k^\times_{n,{\mathfrak p}_n})=1$). 
Consider $t$ in the profinite completion 
$\prod_{{\mathfrak p}_n\in S_n} \wh k^\times_{n,{\mathfrak p}_n}$; then one
has the exact sequence \cite[Chap. 1, \S\,a)]{J0}:
$$1 \to {\mathcal K}^\times_{n,\infty} \tooo 
k_n^\times \otimes \Z_p  \mathop{\tooo}^{\wh \iota_n} \hbox{$\prod_{{\mathfrak p}_n\in S_n}$} 
\wh k^\times_{n,{\mathfrak p}_n} \simeq \Z_p^{\order S} \oplus U_k \to 1. $$
Put $t=\wh \iota_n z$, $z \in k_n^\times \otimes \Z_p$; then 
$\wh \iota_n y = \wh \iota_n (z^{1-\sigma_n})$, 
$y = z^{1-\sigma_n} \, y_\infty$, $y_\infty \in {\mathcal K}^\times_{n,\infty}$, 
then $x_\infty = \No_{k_n/k}(y_\infty)$.
We also have $\Hom^1(G_n, {\mathcal K}^\times_{n,\infty})=1$ \cite[Lemme 5]{J0}.
\end{proof}

The fundamental link between ideal norms in ${k_n}/k$ and the torsion 
group ${\mathcal T}_k$ is given, for $n$ large enough, by the following result 
where the ``uniqueness'' are relative to the choices of the $F_n$; 
we say that some numbers $a \in k'^\times \otimes \Z_p$ (depending on $n$) 
are ``close to $1$'' if $\iota a \to 1$ in $U_k$ when $n \to \infty$.

\begin{theorem} \label{t}
Let $n \gg 0$ fixed and let ${\mathfrak A} \in I_{k_n} \otimes \Z_p$ 
(prime-to-$p$ ideal of $k_n$).

\smallskip
\noindent
\quad (i) There exists $\alpha \in k_n'^\times \otimes \Z_p$ such that
$\No_{k_n/k}({\mathfrak A} \,(\alpha)) = 
\No_{k_n/k}({\mathfrak T}) =: {\mathfrak t}$, with
$\big(\frac{H_{k_n}^\pr/k_n}{{\mathfrak T}} \big) \in {\mathcal T}_{k_n}$,
$\big(\frac{H_k^\pr/k}{{\mathfrak t}} \big) \in {\mathcal T}_k$
and $\iota \No_{k_n/k}(\alpha)$ close to $1$.

\smallskip
\noindent
\quad (ii) The representative ${\mathfrak t}$ of the class 
$\No_{k_n/k}(\mathfrak A) \cdot \No_{k_n/k}(k_n'^\times \otimes \Z_p)$, 
does not depend, modulo $\No_{k_n/k}({\mathcal P}_{k_n,\infty}$),
on the tower $\bigcup_j \!F_j$. 
\end{theorem}

\begin{proof}
(i) From Diagram \ref{diagram4} and the properties of Artin symbols, there exist unique ideals 
${\mathfrak T}, {\mathfrak C} \in I_{k_n} \otimes \Z_p$, modulo ${\mathcal P}_{k_n,\infty}$,
such that:
\begin{equation}\label{AC}
{\mathfrak A} =  {\mathfrak T} \! \cdot\! {\mathfrak C} \! \cdot\!  (y_\infty), \hbox{ with }
\Big(\hbox{$\frac{H_{k_n}^\pr/k_n}{{\mathfrak T}}$} \Big) \in {\mathcal T}_{k_n},
\, \Big(\hbox{$\frac{H_{k_n}^\pr/k_n}{{\mathfrak C}}$} \Big) \in \Gamma_n,
\, y_\infty \in {\mathcal K}^\times_{n,\infty}
\end{equation}

 By restriction, the image of $\Gamma_n$ in $\Gamma$
is $\Gamma^{\,p^n}$; thus
$\No_{k_n/k}({\mathfrak C}) = {\mathfrak c}^{p^n}\cdot(x_\infty)$ for
${\mathfrak c} \in I_k \otimes \Z_p$ such that 
$\big(\frac{H_k^\pr/k}{{\mathfrak c}} \big) \in \Gamma$ and
$x_\infty \in {\mathcal K}^\times_{\infty}$; 
but since $H_k^\nr \subseteq F$, the ideal ${\mathfrak c}$ is $p$-principal,
thus ${\mathfrak c} =(c)$, $c \in k'^\times \otimes \Z_p$, and
then, $\No_{k_n/k}({\mathfrak C}) = (c^{p^n}) \!\cdot (x_\infty)$.
 We have from \eqref{AC}:
$$\No_{k_n/k}({\mathfrak A}) = 
\No_{k_n/k}({\mathfrak T}) \cdot (c^{p^n}) \cdot (x_\infty) \cdot \No_{k_n/k}(y_\infty)
=: {\mathfrak t}  \cdot (c^{p^n}) \cdot (x'_\infty), $$
with $\big(\frac{H_k^\pr/k}{{\mathfrak t}} \big) \in {\mathcal T}_k$,
$x'_\infty \in {\mathcal K}^\times_\infty$. 
From Lemma \ref{Ninfty}, since $(x'_\infty)$ is norm of ideal in $k_n/k$,
$x'_\infty = \No_{k_n/k}(y'_\infty)$, whence
$\No_{k_n/k} \big({\mathfrak A} \, (c)^{-1}\, (y'_\infty)^{-1}\big) = {\mathfrak t}$. 

\smallskip
Let $\alpha= c^{-1}\, y'^{-1}_\infty$; then 
$\iota \No_{k_n/k} (\alpha) = \iota (c^{-p^n})$ is close to $1$.

\smallskip
\noindent
\quad (ii) Let $\bigcup_j F'_j$ be another tower for Diagram \ref{diagram4}; with 
obvious notations (which depend on $n$),
put $u := \No_{k_n/k}(\alpha)$, $u' := \No_{k_n/k}(\alpha')$,
$\iota u$, $\iota u'$ close to $1$, we get
$\No_{k_n/k}({\mathfrak A}) \! \cdot\!  (u) = {\mathfrak t}$,
$\No_{k_n/k}({\mathfrak A}) \! \cdot\!  (u') =  {\mathfrak t}'$. 
Whence ${\mathfrak t}' \, {\mathfrak t}^{-1} = (a)$, with $a$
close to $1$. 
So, if $p^e$ is the exponent of ${\mathcal T}_k$, we obtain
$(a)^{p^e} = (a_\infty) \in {\mathcal P}_{k,\infty}$, which gives
$a^{p^e} = \varepsilon \, a_\infty$, $\varepsilon \in E_k \otimes \Z_p$ with
$\iota \varepsilon$ close to~$1$, hence (for $n \gg 0$) of the form $\varepsilon = \eta^{p^e}$, 
$\eta \in E_k \otimes \Z_p$, with $\iota \eta$ close to $1$ (from Leopoldt's conjecture 
\cite[Theorem III.3.6.2\,(iv)]{Gra1}). 
This yields $(a\,\eta^{-1})^{p^e} = a_\infty$ and we get 
$\iota(a\,\eta^{-1}) = \xi \in W_k = \tor_{\Z_p} (U_k)$; both $\iota a$ and 
$\iota \eta$ are close to $1$, thus $\xi=1$ and $a\,\eta^{-1}=a'_\infty$ 
giving ${\mathfrak t}' \, {\mathfrak t}^{-1} = (a'_\infty) \in \No_{k_n/k}({\mathcal P}_{k_n,\infty})$,
using Lemma \ref{Ninfty}
\end{proof}

\subsection{Images in ${\mathcal C}_k$ and ${\mathcal R}_k$ 
of the ideals ${\mathfrak t}$ -- Conjecture}

We choose, once for all, a set 
$\tk = \big\{\, {\mathfrak t}_\ell \, \big \}_{1 \leq \ell  \leq \order {\mathcal T}_k}$, 
of ideals ${\mathfrak t}_\ell \in I_k \otimes \Z_p$ whose 
Artin symbols describe ${\mathcal T}_k$ isomorphic to 
$\tk \!\cdot\! {\mathcal P}_{k,\infty} / {\mathcal P}_{k,\infty}$.

\smallskip
The ideals $\No_{k_n/k}({\mathfrak A}\,(\alpha)) = {\mathfrak t} \in \tk$,
well-defined modulo $\No_{k_n/k}({\mathcal P}_{k_n,\infty})$,
play the following roles in the evolution of the class and norm factors:

\smallskip
(i) {\it Class factors and $\tk$}.
The ideal groups $\No_{k_n/k}(J_n^i)$, representing the class groups
$\No_{k_n/k}(M_n^i)$ as denominator of the class factors, 
are generated, modulo principal ideals $(a)$, 
$a \in \No_{k_n/k}(k_n'^\times \otimes \Z_p)$, by ideals ${\mathfrak t}^i \in \tk$.

\smallskip
(ii) {\it Norm factors and $\tk$}.
The groups $\Lambda_n^i\! =\! \{x \in k^\times, (x) \in \No_{k_n/k}(J_n^i) \}$,
giving the norm factors, are obtained, modulo elements of
$\No_{k_n/k}(k_n'^\times \otimes \Z_p)$, via principal ideals 
$(\tau) \in \tk$ (hence $\tau$ is local norm at the tame places 
in $k_n/k$ and its norm properties only depend on $S$). 

\smallskip
Put $\tk^\ppl := \{{\mathfrak t} \in  \tk,\  {\mathfrak t} = (\tau)\}$;
the subgroup $\tk^\ppl \cdot {\mathcal P}_{k,\infty}/{\mathcal P}_{k,\infty}$,
is isomorphic to $\Gal(H_k^\pr/ k_\infty H_k^\nr)$.
Let ${\mathfrak t} = (\tau) \in \tk^\ppl$; so we have
$\big(\frac{H_k^\pr/k}{{\mathfrak t}} \big) \in
\Gal(H_k^\pr/k_\infty H_k^\nr) \simeq U_k^*/ \ov E_k$. This yields
$\tau^{p^e} = \varepsilon \, x_\infty$, $\varepsilon \in  E_k \otimes \Z_p$; 
whence $\iota \No_{k/\Q}(\tau) = \pm 1$ and
the image of $\iota \tau$ modulo $\ov E_k$ is defined in $U_k^*/\ov E_k$.
We consider the image of $\log(\iota \tau)$ in 
$\log(U_k^*)/\log(\ov  E_k) = {\mathcal R}_k$,
which defines $\log({\mathfrak t}) := \log(\iota \tau) \pmod{\log(\ov  E_k)}$.
We have $W_k = \Ker(\log)$, and this gives again the exact sequence \eqref{mu}.

\begin{remark}\label{locnorm}
Let $(\tau) \in \tk^\ppl$; choosing a representative of $\tau$ modulo 
${\mathcal K}_\infty^\times$ one may always assume that 
$(\tau) = \No_{k_n/k}({\mathfrak T})$, ${\mathfrak T} \in I_{k_n} \otimes \Z_p$,
since $\No({\mathcal T}_{k_n})={\mathcal T}_k$
(whence $\tau$ local norm at the tame places).
Suppose that the image of $\iota \tau$ in $\tor_{\Z_p}(U_k/\ov E_k)$ is 
in the subgroup $\Gal(H_k^\pr/H_k^\gen)$ generated by the inertia groups
$\tor_{\Z_p}(U_{\mathfrak p} \ov E_k/\ov E_k)$, ${{\mathfrak p} \in S}$; then 
$\tau$ is local norm on $S$. Indeed, let $\iota \tau = u = (u_{\mathfrak p},1,\ldots,1)$;
$u$ is local norm at each ${\mathfrak p}' \ne {\mathfrak p}$, whence 
a global norm (product formula). This explains that generators $\tau$ of 
ideals ${\mathfrak t} \in  \tk^\ppl$, whose images are in 
$\Gal(H_k^\pr/H_k^\gen)$, do not modify any norm factor, only depending 
on the image in ${\mathcal R}_k^\nr = \Gal(H_k^\gen/k_\infty H_k^\nr)$.
\end{remark}

\subsection{The algorithm in terms of fundamental ideals ${\mathfrak t}$.}
We still assume Conventions \ref{conventions} to the base field $k$.
In this subsection, we consider the layer $K=k_n$ 
(with $p^n \gg p^e$, the exponent of~${\mathcal T}_k$)
and, to simplify, we delete indices $n$  
(e.g., $M_n^i \to M^i$, $\Lambda_n^i \to \Lambda^i$, 
$\No_{k_n/k} \to \No$, $b_n \to b$ (number of steps in $K$)); 
then uppercase (respectively lowercase) letters
for ideals are reserved to $K$ (respectively $k$).

\smallskip
From Theorem \ref{t}\,(i) and for any prime-to-$p$ ideals
${\mathfrak A} \in J^i$, defining $M^i$, there exist
$\alpha \in K'^\times \otimes \Z_p$ and ${\mathfrak T}$ of finite 
order modulo ${\mathcal P}_{K,\infty}$, such that
$\No({\mathfrak A} \,(\alpha)) = \No({\mathfrak T}) =: {\mathfrak t} \in \tk$
with $\No(\alpha)$ close to $1$.
Denote by $\TK^i$, the set of such representatives ${\mathfrak T}^i$ 
and let $\Tk^i$ be the set of ${\mathfrak t}^i := \No({\mathfrak T}^i )$; so:
\begin{equation}\label{Sigma}
 \No(\TK^i) =  \Tk^i , \ \ 
\No(M^i)=  \cl_k \langle \Tk^i \rangle, \ \,  \Tk^i \subseteq \tk .
\end{equation} 

Replacing ${\mathfrak A}$ by ${\mathfrak A}\,(\alpha)$ does not modify the 
class and norm factors \eqref{cr} since $\cl_k(\No({\mathfrak A}))
= \cl_k({\mathfrak t})$ and, if 
$\No({\mathfrak A})$ is principal, then ${\mathfrak t} = (\tau)$ is 
equal to $\No({\mathfrak A})$ up to $\No(K^\times \!\otimes \Z_p)$,
which does not modify the norm properties in $\Lambda^i$. 
Then, in $\Lambda^i = \{\tau \in k'^\times \otimes \Z_p, \ (\tau) \in \langle \Tk^i \rangle\}$, 
one must find all elements $\tau^i$ (by definition of the form 
$\No({\mathfrak T}^i)$, ${\mathfrak T}^i \in \langle \TK^i \rangle$),
such that $\tau^i$ is local norm on $S$ in $K/k$, thus of the form
$\No (y^i)$, $y^i \in K'^\times \otimes \Z_p$; so the algorithm continues, from
$\No (y^i) = \No({\mathfrak T}^i)$, with the following evolution using Theorem \ref{t}\,(i)):
\begin{equation}\label{evolution}
(y^i) = {\mathfrak T}^i \cdot  {\mathfrak B}^{1-\sigma}, \ 
\hbox{with $\No( {\mathfrak B} \,(\beta))= \No({\mathfrak T}')
={\mathfrak t}'  \in \tk$,} 
\end{equation}
for a suitable $\beta$ such that $\No(\beta)$ is close to $1$,
and one obtains a new ${\mathfrak t}'$ to build $\Tk^{i+1}$, and so on.
If $\lambda$ or $\mu$ do not vanish, there exist, when $[K : k] \to\infty$, 
arbitrary large $i$-sequences of sets $\Tk^i$ such that the class and norm factors 
are constant, which seems incredible, each new ${\mathfrak t}'$ being a priori 
random in~$\tk$. 

A philosophy should be that it is the ${\mathfrak t}'$ which govern (numerically) 
the $G$-structure of the class groups in $K/k$ and not the inverse (see also 
\cite[Remarques 11, \S\,6]{Gra4}). 

\smallskip
Let's give a more precise description of the numerical possibilities, assuming 
to simplify the comments that $1 \to {\mathcal R}_k \to {\mathcal T}_k \to {\mathcal C}_k \to 1$
is an exact sequence of $\F_p$-vector spaces; we compute
the filtration $\{M^i\}_{i \geq 0}$ for $M= {\mathcal C}_K$ ($K=k_n$ fixed)
with the following exact sequence at the step $i$ (see \eqref{Sigma}):
\begin{equation*}
1 \too \Lambda^i/E_k \mathop{\tooo}^{(\,.\,)}  \langle \Tk^i  \rangle 
\mathop{\tooo}^{\cl_k}  \cl_k \langle \Tk^i \rangle = \No(M^i) \too 1,
\end{equation*}
where $\Lambda^i = \{\tau \in k'^\times \otimes \Z_p,\ (\tau) \in \langle \Tk^i \rangle\}$,
and let ${\mathfrak t}^{i+1}$ (obtained as above). Various cases
may arrive to get the $(i+1)$th exact sequence
\begin{equation*}
1 \to \Lambda^{i+1}/E_k \mathop{\tooo}^{(\,.\,)} \langle \Tk^{i+1} \rangle 
\mathop{\tooo}^{\cl_k} \cl_k \langle \Tk^{i+1} \rangle  = \No(M^{i+1}) \to 1:
\end{equation*}

\noindent
\quad $\bf (a)$ $\cl_k({\mathfrak t}^{i+1}) \notin \cl_k \langle\Tk^i \rangle$.
Thus $\No(M^{i+1}) \supsetneq \No(M^i)$ and
this decreases the class factor; but there is no new relation of
principality between ideals, so $\Lambda^{i+1} = \Lambda^i$
(norm factor unchanged).

\smallskip
\noindent
\quad $\bf (b)$ $\cl_k({\mathfrak t}^{i+1}) \in \cl_k \langle \Tk^i \rangle$.
Thus $\No(M^{i+1}) = \No(M^i)$ (class factor unchanged);
but ${\mathfrak t}^{i+1} = (\tau) \cdot \prod_j {\mathfrak t}^i_j{}^{a_j}$
gives, possibly, some $\tau \notin \Lambda^i$. Then two cases arise: 

\smallskip
\noindent
\qquad $\bf (i)$ $\tau \notin \Lambda^i \, \No (K^\times)$, therefore
$(\Lambda^{i+1} : \Lambda^{i+1} \cap \No (K^\times)) >
(\Lambda^i : \Lambda^i \cap \No (K^\times))$,
which decreases the norm factor.

\smallskip
\noindent
\qquad $\bf (ii)$  $\tau \in \Lambda^i\, \No (K^\times)$ (class and norm factors 
unchanged). This is the ``bad case'' occurring, roughly, $O(\lambda n +\mu p^n)$ 
times if Greenberg's conjecture falls (see Remark \ref{locnorm} for more enlightenment). 

\smallskip
We have given, in  \cite[Section 6]{Gra4}, some heuristics about the ``equation
$(y) = {\mathfrak A} \, {\mathfrak B}^{1-\sigma}\,$'' in cyclic extensions $L/K$ when 
$\No_{L/K}(y) = \No_{L/K}({\mathfrak A})$ and its ``additive aspects'', which applies to
$(\tau) = \No( {\mathfrak T}) = \No(y)$ and $(y) = {\mathfrak T} \, {\mathfrak B}^{1-\sigma}$.

\smallskip
Assuming that the ideals ${\mathfrak t}$, given by the algorithm,
are random, $\cl_k({\mathfrak t})$ (resp. $\log(\iota \tau) \pmod {\log(\ov E_k)}$ 
are random in ${\mathcal C}_k$ (resp. ${\mathcal R}_k$). This is likely to avoid 
unbounded algorithms and suggests the following conjecture:

\begin{conjecture}\label{conj}
For $n \gg 0$ fixed, let ${\mathfrak t}_j$ (or
$\tau_j$, when ${\mathfrak t}_j = (\tau_j)$ is principal), be the fundamental 
ideals encountered by the algorithm computing inductively the successive 
class and norm factors, in $b_n$ steps; then:

\smallskip
\noindent
\quad (i) The classes $\cl_k({\mathfrak t}_j)$ are uniformly 
distributed in ${\mathcal C}_k$.

\smallskip
\noindent
\quad (ii) When ${\mathfrak t}_j = (\tau_j)$, the images
$\log({\mathfrak t}_j) := \log(\iota \tau_j) 
\pmod{\log(\ov  E_k)}$ are uniformly distributed in the normalized 
regulator ${\mathcal R}_k = \log(U_k^*)/\log(\ov  E_k)$.
\end{conjecture}

\subsection{Conclusion and possible methods}
Recall that $b_n$ is the length of the algorithm for the layer $n$.
We observe the huge discontinuity between the case $b_n$ bounded,
which characterizes Greenberg's conjecture (Theorem \ref{O} and 
Corollary \ref{bn=1}) and the case where 
$\lambda$ or $\mu$ are non-zero, giving $b_n \to \infty$. 
In other words, there is a conflict between the ``random aspect'' of the 
algorithm, when $\lambda$ or $\mu$ are non-zero, and the smooth algebraic 
form given by Iwasawa's theory.
We indeed have, under Conventions \ref{conventions}, 
$\order {\mathcal C}_{k_n} = p^{\lambda\, n + \mu \,p^n+ \nu}$ 
for all $n \geq 0$, so that the algorithm must obtain rigorously 
these formulas, {\it for all $n$}, which seems to be an excessive 
requirement in contradiction with Conjecture \ref{conj}.

\smallskip
To give a logical way, the sole ``solution'', where $b_n$ does not tend 
to infinity, is $b_n$ constant for all $n \geq n_1$, giving, from
the new base field $k_{n_1}$, that we still denote $k$, the well-known 
properties when Greenberg's conjecture holds.
In that case, ${\mathcal C}_{k_n}^2/{\mathcal C}_{k_n}^1=1$ and $b_n\leq 1$ for all $n$.
In other words, in this situation, the ``unpredictable'' evolution relation
\eqref{evolution} is not needed.
The quotient ${\mathcal C}_{k_n}^2/{\mathcal C}_{k_n}^1$ does appear (written instead 
$(1-\sigma) {\mathcal C}_{k_n} [(1-\sigma)]$) in works of Koymans--Pagano--Smith 
\cite{KP,Sm}, where deep distribution results are proved for the degree $p$ cyclic case.

\smallskip
We believe that these techniques can be successful 
for Greenberg's conjecture since the general algorithm 
of ``unscrewing'' in $k_n/k$ is identical and is essentially based 
on random values of classical norm symbols.
In other words, Greenberg's conjecture 
would be, for $k_\infty/k$ ($k$ taken hight enough in the cyclotomic tower),
an extreme version (of the degree $p$ cyclic case) giving the non-existence 
of ``exceptional $p$-classes'' (i.e., non-invariant $p$-classes) in the tower, 
that is to say, $b_n \leq 1$ for all $n \geq 0$ (to be compared with $b_n \to \infty$
if $\lambda$ or $\mu$ do not vanish).

\begin{remark}
For a base field which does not fulfill the previous conditions, 
the algorithms may need several steps and (under Greenberg's conjecture) they
regularize at some layer such that the above trivialization holds; for instance,
the case of $k=\Q(\sqrt{6559})$, $p=3$, computed in \cite[\S\,7.2]{Gra4}, 
yields ${\mathcal C}_k \simeq \Z/9\Z$, ${\mathcal R}_k \simeq \Z/27\Z$, 
${\mathcal C}_{k_1} \simeq \Z/27\Z \times \Z/3\Z$ (whence $b_1=2$) and
${\mathcal C}_{k_2} \simeq \Z/27\Z \times \Z/9\Z$; we compute with \cite{BJ} 
that $\wt {\mathcal C}_k \simeq \Z/3\Z$, $\wt {\mathcal C}_{k_1} \simeq 
\wt {\mathcal C}_{k_2} \simeq \Z/9\Z$.
\end{remark}

All this shows how classical arguments of {\it algebraic number theory} seem insufficient 
to prove unconditionally Greenberg's conjecture (among others), but that density 
results may be accessible, giving that the conjecture holds except, possibly, for 
pathological families of zero density (probably none).


\end{document}